\documentclass[11pt,a4paper,reqno,dvipsnames]{amsart}
\usepackage[english]{babel}
\usepackage[T1]{fontenc}
\usepackage{verbatim}
\usepackage{palatino}
\usepackage{amsmath}
\usepackage{amssymb}
\usepackage{tcolorbox}
\usepackage{amsthm}
\usepackage{amsfonts}
\usepackage{graphicx}
\usepackage{esint}
\usepackage{tikz}
\usepackage{color}
\usepackage{xcolor}
\usepackage{mathtools}
\usepackage{overpic}

\usepackage[colorlinks = true, citecolor = black]{hyperref}
\usepackage{cleveref}
\author{Benjamin Jaye and Rahul Sethi}
\title[A High-Frequency Uncertainty Principle for the Fourier-Bessel Transform]{A High-Frequency Uncertainty Principle for\\ the Fourier-Bessel Transform}
\address{School of Mathematics\\ Georgia Institute of Technology \\ Atlanta, GA, USA} \email{bjaye3@gatech.edu}
\email{rahul.sethi@math.gatech.edu}

\date{\today}
\keywords{Uncertainty Principle, Fourier-Bessel Transform, Logvinenko-Sereda Theorem}

\newcommand{\supp}{\operatorname{supp}}

\newcommand{\C}{\mathbb C}

\newcommand{\N}{\mathbb N}

\newcommand{\F}{\mathcal{F}}
\newcommand{\ep}{\varepsilon}

\newcommand{\R}{\mathbb R}

\newcommand{\wh}{\widehat}

\renewcommand{\ge}{\geqslant}
\renewcommand{\le}{\leqslant}

\newcommand{\inv}{^{-1}}

\makeatletter
\newcommand{\citepre}[2]{%
  \begingroup
    \renewcommand\@cite[2]{[#1, ##1]}%
    \cite{#2}
  \endgroup
}
\makeatother

\numberwithin{equation}{section}

\theoremstyle{plain}
\newtheorem{theorem}[equation]{Theorem}
\newtheorem*{"thm"}{"Theorem"}

\newtheorem{lemma}[equation]{Lemma}

\newtheorem{proposition}[equation]{Proposition}
\newtheorem{question}{Question}

\theoremstyle{definition}

\theoremstyle{remark}

\newcommand{\nref}[1]{(\hyperref[#1]{#1})}

\begin{document}

\begin{abstract}

Motivated by problems in control theory concerning decay rates for the damped wave equation 
$$w_{tt}(x,t) + \gamma(x) w_t(x,t) + (-\Delta + 1)^{s/2} w(x,t) = 0,$$
we consider an analogue of the classical Paneah-Logvinenko-Sereda theorem for the Fourier Bessel transform. In particular, if $E \subset \mathbb{R}^+$ is $\mu_\alpha$-relatively dense (where $d\mu_\alpha(x) \approx x^{2\alpha+1}\, dx$) for $\alpha > -1/2$, and $\operatorname{supp} \mathcal{F}_\alpha(f) \subset [R,R+1]$, then we show 
$$\|f\|_{L^2_\alpha(\mathbb{R}^+)} \lesssim \|f\|_{L^2_\alpha(E)},$$
for all $f\in L^2_\alpha(\mathbb{R}^+)$, where the constants in $\lesssim$ do not depend on $R > 0$.

Previous results on PLS theorems for the Fourier-Bessel transform by Ghobber and Jaming (2012) provide bounds that depend on $R$. In contrast, our techniques yield bounds that are independent of $R$, offering a new perspective on such results. This result is applied to derive decay rates of radial solutions of the damped wave equation.  

 \end{abstract}

\maketitle


\section{Introduction}
The motivation for this paper centers around the following basic question:

\begin{question}\label{genann} Suppose that $R>1$ and $f\in L^2(\R^n)$ has its Fourier transform supported in the width-one annulus $A_R = \{\xi\in \R^n:|\xi|\in [R, R+1]\}$.  For which sets $E\subset \R^n$ does it hold that
\begin{equation}
\label{pls-independent}
    \|f\|_{L^2(\R^n)}\leq C\|f\|_{L^2(E)}
\end{equation}
where $C>0$ is independent of $f$ and $R$?
\end{question}

Since the width-one annulus contains a strip of length approximately $\sqrt{R}$ and width $1/2$ in any direction, one can show (see \cite{green-jaye-mitkovski}) that a necessary condition for $E$ is that it satisfies the \emph{geometric control condition} (GCC):  there exists $L>0$ and $\theta>0$ such that for any line segment $\ell\subset \R^n$ of length $L$, we have $m(E\cap \ell)\geq \theta\cdot L$. Here $m$ is the $1$-dimensional Lebesgue measure on the segment $\ell$.  The principal question is whether the GCC is sufficient:

\begin{question}\label{gccques} Suppose that $E$ satisfies the GCC.  Is it true that there exists $C>0$ such that for every $R>1$ and $f\in L^2(\R^n)$ with $\supp(\wh{f})\subset A_R$ we have
$$\|f\|_{L^2(\R^n)}\leq C\|f\|_{L^2(E)}?
$$ \end{question}

The motivation for considering this problem comes from a problem in control theory considered by Burq-Joly \cite{burq16}.  For a non-negative function $\gamma$ we consider the damped wave equation: For $(x,t) \in \R^n \times [0,\infty)]$, let $w$ satisfy
 	\begin{equation}\label{eq:1} w_{tt}(x,t) +\gamma(x)w_t(x,t) + \Delta w(x,t)=0.\end{equation}
The damping force is represented by $\gamma w_t$, and the energy of the solution $w$ is given by $E(t) = \|w\|_{H^1(\R^n)}+\|\partial_t w\|_{L^2(\R^n)}$.
Standard analysis shows that if $\gamma=0$, then the energy is conserved, i.e., there is no decay. On the other hand, for constant damping $\gamma=c >0$, it can be shown that $E(t)$ decays exponentially in $t$.  A classical question is to understand for which functions $\gamma$ does $E(t)$ decay exponentially. In \cite{burq16} it is proved for \emph{uniformly continuous} $\gamma$, that the energy decays exponentially if there exists $\ep>0$ such that the level set $\{\gamma>\ep\}$ satisfies the GCC.  They asked whether one can drop the assumption of uniform continuity, which would follow from Question \ref{gccques}.  This application was the motivation for Walton Green, Mishko Mitkovski, and the first author to formulate Question \ref{gccques} in \cite{green-jaye-mitkovski}, where it was proved that the answer to Question \ref{gccques} is affirmative if one replaces $E$ by its $\delta$-neighborhood for any $\delta>0$ (with a constant that blows up as $\delta\to 0^+$, see \cite{green-jaye-mitkovski}).  Analogues of the GCC in compact manifolds have long played a role in  control theory for hyperbolic equations, for instance see Bardos, Lebeau, Rauch \cite{b-l-r, zworski-book}, and various uncertainity principles have played an essential role in control theory \cite{bourgain18,jaffard01,zworski-book,suzuki-inami,suzuki-damping}.

In this paper we consider the radial case of Question \ref{gccques} and show that it has a positive answer in this case, even with a considerably weaker condition than the GCC. 

For $\alpha > -1/2$, the Fourier-Bessel transform is defined by $$\mathcal F_\alpha(f)(y) = \int_0^\infty f(x) j_\alpha(2\pi xy)\, d\mu_\alpha(x),$$
      where $j_\alpha$ is the modified Bessel function given by $$j_\alpha(x) = \Gamma(\alpha + 1) \sum_{n=0}^\infty \frac{(-1)^n}{n!\, \Gamma(n + \alpha + 1)} \left(\frac{x}{2}\right)^{2n} = \frac{\Gamma(\alpha + 1)}{\Gamma\left(\alpha + \frac12\right) \Gamma\left(\frac12\right)} \int_{-1}^1 (1-x^2)^{\alpha - 1/2}\, \cos(sx)\, dx,$$
      and
      $$d\mu_\alpha(x) = \frac{2\pi^{\alpha + 1}}{\Gamma(\alpha + 1)}\, x^{2\alpha+1}\, dx.$$

The Fourier-Bessel transform occurs naturally when one studies the Fourier transforms of radial functions. If $f: \R^d \to \R$ is a radial function, i.e., $f(x) = F(|x|)$, then a standard calculation shows that $\widehat{f}(\xi) = \F_\alpha (F)(|\xi|)$ for $\alpha = \frac{d}{2} - 1$.
      
      For $1\leqslant p < \infty$, $L^p_\alpha(\R^+)$ consists of the measurable functions $f:\R^+ \to \R$ such that
      $$\|f\|_{L^p_\alpha} := \Bigl(\,\int_0^\infty |f(x)|^p\, d\mu_\alpha(x) \Bigl)^{1/p}<\infty.$$
      The Fourier-Bessel transform extends to an isometry on $L^2_\alpha(\R^+)$, i.e., $\|\mathcal{F}_\alpha(f)\|_{L^2_\alpha} = \|f\|_{L^2_\alpha}$.

With this notation, our main result reads as follows:

      \begin{tcolorbox}[colback=SeaGreen!10, colframe=SeaGreen!80, boxrule=0.5mm, left=2mm, right=2mm, boxsep=1mm, arc=1mm.]
\begin{theorem}
\label{thm1}
    Suppose $E\subset\R^+$ is relatively dense with respect to $\mu_{\alpha}$, meaning that
$$\text{there is }\gamma>0 \text{ such that }\mu_\alpha(E\cap [r,r+1]) \ge \gamma \mu_\alpha([r,r+1]) \text{ for all }r\ge 0.$$ 
There exists $C>0$ such that for all sufficiently large $R > 0$, if $\supp \F_\alpha(f) \subset [R,R+1]$, then 
$$\|f\|_{L^2_\alpha(\R^+)} \le C\|f\|_{L^2_\alpha(E)},$$
where $C$ depends on $\alpha$ and $\gamma$ but not on $R$.
\end{theorem}
\end{tcolorbox}

If $E\subset \R^d $ satisfies the GCC, then the spherical average of $E$  
$$g(r) = \int_{\mathbb{S}^{d-1}}\chi_E(r\omega)d\omega$$
has a superlevel set that is relatively dense\footnote{In fact much stronger conditions are a consequence of the GCC:  The function $g(r)$ has a superlevel set $F$ with the property that there is a constant such that $\mu_{\alpha}(F\cap(x, x+C\min(1,1/x))\gtrsim \mu_{\alpha}(x,x+C\min(1,1/x))$ for all $x\geq 0$.  This condition means that $F$ is ``thick" (it is the complement of a thin set in the sense of \cite{shubin98}).  It was not initially clear to us whether the `thickness' would be required to prove Theorem \ref{thm1} below and it was somewhat surprising that relative density was sufficient.}  with respect to $\mu_{\alpha}$, with $\alpha = \frac{d}{2}-1$.  Therefore, a corollary of Theorem \ref{thm1} is that Question \ref{gccques} has a positive solution for radial functions. 

Theorem \ref{thm1} is reminiscent of the Paneah-Logvinenko-Sereda uncertainty principle \cite{havin12, muscschlag,logvinenko74,paneah61}.  Ghobber and Jaming \cite{gj} proved an analogue of the Paneah-Logvinenko-Sereda uncertainty principle for functions whose Fourier-Bessel support is contained in an interval $[0,R]$ for some $R > 0$, with the same assumptions on $E$. Modifying some of their notation, we state their result below. Combining \Cref{thm-gj} with our result, we note that \Cref{thm1} holds for all $R >0$ when $\alpha \ge 0$. 

\begin{theorem}
\label{thm-gj}
    Let $\alpha \ge 0$ and let $R, \gamma > 0$. Let $f\in L^2_\alpha(\R^+)$ such that $\operatorname{supp} \F_\alpha (f) \subset [0,R]$. If $E$ is a $\gamma$-relatively dense set (with constant $\gamma>0$), then 
    $$\|f\|^2_{L^2_\alpha(\R^+)} \le \frac{3}{2} \left(\frac{300 \cdot 9^\alpha}{\gamma} \right)^{\frac{160\sqrt{3} \pi}{2\ln 2} R + \alpha \frac{\ln 3}{\ln 2} + 1}\, \|f\|^2_{L^2_\alpha(E).}$$
\end{theorem}

Observe that the constant in Theorem \ref{thm-gj} grows exponentially with the parameter $R$ (and necessarily so). Consequently, a naïve application of Theorem \ref{thm-gj} to the interval $[0,R+1]$ to treat the case $\operatorname{supp} \F_\alpha (f) \subset [R,R+1]$ would yield estimates that deteriorate as $R$ grows. On the other hand, we obtain estimates independent of $R$. While Theorem \ref{thm-gj}'s proof hinges on the fact that Bessel functions are entire of exponential type, our key ingredient is that these same Bessel functions are well-approximated by trigonometric polynomials.

We now describe briefly how Theorem \ref{thm1} applies to understanding energy decay rate for the fractional damped wave equation (\ref{damped-wave-eqn}) introduced by Malhi and Stanislavova \cite{malhi2020energy}. Fix $s > 0$ and a damping function $\gamma: \R^d \to [0, \infty)$. For every $(x,t) \in \R^d \times [0, \infty)$, let $w$ satisfy
\begin{equation}
    \label{damped-wave-eqn}
    w_{tt}(x,t) + \gamma(x) w_t(x,t) + (-\Delta + 1)^{s/2} w(x,t) = 0.
\end{equation}
      The damping force is represented by $\gamma w_t$ and the fractional Laplacian is defined, for $r \in \R$, by 
      $$(-\Delta + 1)^r f(x) := \int_{\R^d} (|\xi|^2 + 1)^r\, \widehat{f}(\xi)\, e^{2\pi ix\xi}\, d\xi.$$
      We study the decay rate of the energy $E(t)$, defined by
      $$E(t) := \|(w(t), w_t(t))\|_{H^{s/2} \times L^2} = \left(\int_{\R^d} |(-\Delta + 1)^{s/4}\, w(x,t)|^2 + |w_t(x,t)|^2 \,dx \right)^{1/2}.$$

As in \cite{malhi2020energy, green2019decay, green-jaye-mitkovski}, we can derive from Theorem \ref{thm1} a resolvant estimate from which semigroup theory \cite{gearhart78, pruss84, borichev10} yields the following theorem, which in particular answers Burq and Joly's question about the removal of continuity of the damping function positively in the radial case. Their question remains open in the general case. 

\begin{theorem}
    \label{thm-damped}
     Let \(d\ge1\), \(s>0\), and \(\alpha=\tfrac d2-1\), and let \(\gamma:[0,\infty)\to[0,\infty)\) be a bounded radial damping function.  Suppose there exist constants \(c_{0}, \gamma_1 >0\), and \(L>0\) such that the set \(\,E=\{\,r\ge0:\gamma(r)\ge c_{0}\}\) 
satisfies the density condition
\begin{equation}
\label{rdensity}
    \mu_{\alpha}(E\cap[R,R+L])\ge\gamma_{1}\,\mu_{\alpha}([R,R+L]),
\end{equation}
for every \(R\ge0\). Then, any radial solution \(u(r,t)\) of 
\(u_{tt}+(-\Delta+1)^{s/2}\, u+\gamma(r)\,u_t=0\) 
in \(\R^{d}\), with initial data 
\((u(\cdot,0),u_t(\cdot,0))\in H^{s}_{\mathrm{rad}}(\R^{d})\times H^{s/2}_{\mathrm{rad}}(\R^{d})\) 
and finite energy \(E(0)\), satisfies
\[
E(t)\;\le\;
\begin{cases}
C\,(1+t)^{-\frac{s}{4-2s}}\;\bigl\|\,u(0),u_t(0)\bigr\|_{H^{s}\times H^{s/2}},
&0<s<2,\\[0.5em]
C\,e^{-\omega\,t}\;E(0),
&s\ge2,
\end{cases}
\]
for all \(t\ge0\), where \(C,\omega>0\) depend only on \(d,s,c_{0},\gamma_{1}\).
\end{theorem}

Finally, comparing Theorem \ref{thm1} and Theorem \ref{thm-gj}, it is natural to ask whether one can  also prove a version of Theorem \ref{thm1} for functions with Fourier-Bessel transform supported in a finite union of disjoint unit intervals.  Our next result shows that this is the case under the additional assumption $\alpha - \frac{1}{2} \in \N$.  We suspect that Theorem \ref{thm2} below holds for all $\alpha > -\frac{1}{2}$.

\begin{tcolorbox}[colback=SeaGreen!10, colframe=SeaGreen!80, boxrule=0.5mm, left=2mm, right=2mm, boxsep=1mm, arc=1mm.]
\begin{theorem}
\label{thm2}
Suppose $\alpha - \frac{1}{2} \in \N$, and $E\subset\R^+$ is relatively dense with respect to $\mu_{\alpha}$, meaning that
$$\text{there is }\gamma>0 \text{ such that }\mu_\alpha(E\cap [r,r+1]) \ge \gamma \mu_\alpha([r,r+1]) \text{ for all }r\ge 0.$$ 
Then, if $\supp \F_\alpha(f) \subset \bigcup_{k=1}^N I_k$, where the $I_k$'s are disjoint intervals of size $1$, then 
$$\|f\|_{L^2_\alpha(\R^+)} \lesssim \|f\|_{L^2_\alpha(E)},$$
where the constants in $\lesssim$ depend on $\alpha, \gamma, N$, but not on the positions of the intervals.
\end{theorem}
\end{tcolorbox}

 To prove Theorem \ref{thm1}, we first employ the asymptotic expansion of the Bessel function, which allows us to apply Kovrizhkin’s multi-interval version of the Paneah–Logvinenko–Sereda theorem \cite{k1}. By contrast, the proof of Theorem \ref{thm2} does not invoke Kovrizhkin’s theorem directly, though it is closely modelled on his argument. This adaptation is non-trivial: when the interval is in the low-frequency regime, we rely on a Bernstein inequality due to Ghobber and Jaming \cite{gj}, whereas in the high-frequency case we treat the kernel of the Bessel transform (after suitable manipulations) as a weight.

\section{Proof of Theorem \ref{thm1}}

Note that $\mu_\alpha$-relative density and Lebesgue relative density are equivalent conditions on the set $E$. The former asks for $\mu_\alpha(E\cap [r,r+1]) \gtrsim \mu_\alpha([r,r+1])$ for all $r \ge 0$, while the latter asks for $|E \cap [r,r+1] \gtrsim 1$ for all $r \ge 0$. We record the proof in the following proposition.

\begin{proposition}
\label{relative-density-equivalence}
    Let $\alpha > -1/2$. A set $E \subset \R^+$ is $\mu_\alpha$-relatively dense if and only if it is Lebesgue relatively dense. 
\end{proposition}
\begin{proof}
    Suppose $\gamma > 0$ is such that $|E \cap [r, r+1]| \ge \gamma$ for all $r \ge 0$. We will find $\widetilde{\gamma} > 0$ such that $\mu_\alpha(E \cap [r,r+1]) \ge \widetilde{\gamma} \mu_\alpha([r,r+1])$. Using $r^{2\alpha + 1} \le x^{2\alpha + 1} \le (r+1)^{2\alpha + 1},$ we have the inequalities
    $$c_1 r^{2\alpha + 1} |E \cap [r,r+1]| \le \mu_\alpha(E \cap [r,r+1]) \le c_1 (r+1)^{2\alpha + 1}\, |E \cap [r,r+1]|,$$
    and $$c_1 r^{2\alpha + 1} \le \mu_\alpha([r,r+1]) \le c_1 (r+1)^{2\alpha + 1},$$
    where $c_1 := \frac{2\pi^{\alpha + 1}}{\Gamma(\alpha + 1)}$. Combining the two, we get
    $$\frac{\mu_\alpha(E \cap [r,r+1])}{\mu_\alpha([r,r+1])} \ge \gamma \left(\frac{r}{r+1}\right)^{2\alpha + 1}  \ge \gamma \left(\frac{\frac{\gamma}{2}}{\frac{\gamma}{2}+1}\right)^{2\alpha + 1} > 0,$$
    when $r \ge \frac{\gamma}{2}$. If $r \le \frac{\gamma}{2}$, we split $I := [r,r+1]$ as $I = J \cup J'$, where $J = [r + \frac{\gamma}{2}, r + 1]$ and $J' = [r, r+ \frac{\gamma}{2}]$. We have $$\mu_\alpha(E \cap I) \ge \mu_\alpha(E \cap J) = c_1\int_{E \cap J} x^{2\alpha + 1}\, dx \ge c_1 \left(\frac{\gamma}{2} \right)^{2\alpha + 1} |E \cap J| \ge c_1 \left(\frac{\gamma}{2} \right)^{2\alpha + 2},$$
    as
    $$|E \cap J| = |E \cap I| - |E \cap J'| \ge \gamma - \frac{\gamma}{2} = \frac{\gamma}{2}.$$
    On the other hand, $$\mu_\alpha(I) = c_1\int_I x^{2\alpha + 1}\, dx \le c_1 (r+1)^{2\alpha + 1} \le c_1 \left(\frac{\gamma}{2} + 1 \right)^{2\alpha + 1},$$
    giving 
    $$\frac{\mu_\alpha(E \cap [r,r+1])}{\mu_\alpha([r,r+1])} \ge \frac{\gamma}{2} \left(\frac{\frac{\gamma}{2}}{\frac{\gamma}{2}+1}\right)^{2\alpha + 1},$$
    for all $r \le \frac{\gamma}{2}$. We obtain $\mu_\alpha(E \cap [r,r+1]) \ge \widetilde{\gamma} \mu_\alpha([r,r+1])$ with $\widetilde{\gamma} := \frac{\gamma}{2} \left(\frac{\frac{\gamma}{2}}{\frac{\gamma}{2}+1}\right)^{2\alpha + 1}$. The reverse implication is entirely analogous.
\end{proof}

We will use the following multi-interval Logvinenko-Sereda theorem for the Fourier transform, due to Kovrizhkin.
\begin{theorem}[\cite{k1}]
    \label{multi-interval-pls}
    Let $E\subset \R$ satisfy $$|E \cap [r,r+1]| \ge \gamma,$$
    for some constant $\gamma > 0$. 
    Let $\{J_k\}_{k=1}^N$ be intervals with $|J_k| = b$. If $1\le p \le \infty$ and $f\in L^p(\R)$ with $\operatorname{supp} \widehat{f} \subset \bigcup_{k=1}^N J_k$, then 
    $$\|f\|_{L^p(\R)} \le \left(\frac{\gamma}{C}\right)^{-ab \left(\frac{C}{\gamma}\right)^N - N + \frac{p-1}{p}}\, \|f\|_{L^p(E)}.$$
\end{theorem}

\begin{proof}[Proof of Theorem \ref{thm1}]
    Without loss of generality, consider $f\in L^1_\alpha(\R^+) \cap L^2_\alpha(\R^+)$, and assume $\supp \F_\alpha(f) \subset [R,R+1]$ for some $R > 0$. Fourier-Bessel inversion gives
$$f(t) = c_\alpha \int_R^{R+1} \mathcal{F}_\alpha(f)(y) j_\alpha(2\pi ty)\, y^{2\alpha + 1}\, dy,$$
where $c_\alpha = \frac{2\pi^{\alpha + 1}}{\Gamma(\alpha + 1)}.$ As $t\to\infty$, we have the following asymptotic,
$$j_{\alpha}(t) = A_\alpha\,
t^{-\alpha - 1/2} \cos\!\left( t -\delta \right)
+ O(t^{-\alpha - 3/2}),$$
where $A_\alpha =\frac{2^{\alpha + 1/2} \Gamma(\alpha + 1)}{\sqrt{\pi}}$ and $\delta = (2\alpha + 1)\frac{\pi}{4}$.

Through a change of variables $y = R+s$, we get
$$f(t) = c_\alpha\int_0^1 \mathcal{F}_\alpha(f)(R+s) j_\alpha(2\pi(R+s)t)\, (R+s)^{2\alpha + 1}\, ds.$$
So,
$$\|f\|_{L^2_\alpha}^2 = c_\alpha^3 \int_0^\infty \left|\int_0^1 \mathcal{F}_\alpha(f)(R+s) j_\alpha(2\pi(R+s)t)\, (R+s)^{2\alpha + 1}\, ds\right|^2 \, t^{2\alpha + 1}\, dt.$$

We can cut out $[0, \ep]$ for small enough $\ep \in (0,1)$. Indeed, consider
$$I_\ep := \int_0^\ep \left|\int_0^1 \mathcal{F}_\alpha(f)(R+s) j_\alpha(2\pi(R+s)t)\, (R+s)^{2\alpha + 1}\, ds\right|^2 \, t^{2\alpha + 1}\, dt.$$

We will use the estimate
\begin{equation}
    \label{bessel-function-estimate}
    |j_\alpha(t)| \lesssim_\alpha (1+t)^{-\alpha-1/2}
\end{equation}
from \cite{gj}. Since $\alpha > -1/2$, this implies $|j_\alpha(t)| \lesssim_\alpha t^{-\alpha-1/2}$ which is all we need. We have
\begin{align*}
    I_{\ep} &\le \int_0^\ep \left(\int_0^1 |\mathcal{F}_\alpha(f)(R+s)| |j_\alpha(2\pi(R+s)t)|\, (R+s)^{2\alpha + 1}\, ds\right)^2 \, t^{2\alpha + 1}\, dt\\
    &\lesssim \int_0^\ep \left(\int_0^1 |\mathcal{F}_\alpha(f)(R+s)| (R+s)^{\alpha + 1/2}\, ds\right)^2 \, dt\\ 
    &\lesssim \int_0^\ep \int_0^1 |\mathcal{F}_\alpha(f)(R+s)|^2 (R+s)^{2\alpha + 1}\, ds \, dt\\
    &\lesssim \ep\,  \|\F_\alpha(f)\|^2_{L^2_\alpha} \lesssim \ep\, \|f\|^2_{L^2_\alpha},
\end{align*}
using the Cauchy-Schwarz inequality and Plancherel's theorem for $L^2_\alpha(\R^+)$.

Now, consider 
$$J_{1, \ep} := \int_\ep^\infty \left|\int_0^1 \mathcal{F}_\alpha(f)(R+s) ((R+s)t)^{-\alpha-1/2} \cos(2\pi(R+s)t-\delta)\, (R+s)^{2\alpha + 1}\, ds\right|^2 \, t^{2\alpha + 1}\, dt,$$
and
$$J_{2, \ep} := \int_\ep^\infty \left|\int_0^1 \mathcal{F}_\alpha(f)(R+s) ((R+s)t)^{-\alpha-3/2}\, (R+s)^{2\alpha + 1}\, ds\right|^2 \, t^{2\alpha + 1}\, dt.$$
The error term $J_{2, \ep}$ is easily controlled as follows:
\begin{align*}
    J_{2,\ep} &= \int_\ep^\infty \frac{dt}{t^2} \int_0^1 |\F_\alpha(f)(R+s)|^2 (R+s)^{-2} (R+s)^{2\alpha+1}\, ds\\
    &\le \frac{1}{\ep R^2} \int_0^1 |\F_\alpha(f)(R+s)|^2 (R+s)^{2\alpha+1}\, ds \lesssim \frac{\|f\|^2_{L^2_\alpha}}{\ep R^2}.
\end{align*}

Now, we focus on $J_{1, \ep}$. The powers of $t$ cancel to give
$$J_{1,\ep} = \int_\ep^\infty \left|\int_0^1 \mathcal{F}_\alpha(f)(R+s) (R+s)^{\alpha+1/2} \cos(2\pi(R+s)t-\delta)\, ds\right|^2 \, dt.$$
For convenience, let $h(s) := \mathcal{F}_\alpha(f)(R+s) (R+s)^{\alpha+1/2}$ so that $\int_0^\infty |h(s)|^2\, ds \approx \|f\|^2_{L^2_\alpha}$ and $\operatorname{supp} h \subset [0,1]$. Also, define $$g(t) := \int_0^1 h(s) \cos(2\pi(R+s)t - \delta)\, ds.$$ Then, $J_{1,\ep} = \int_\ep^\infty |g(t)|^2\, dt$. Writing $\cos(2\pi Rt-\delta + 2\pi st) = \cos(2\pi Rt-\delta) \cos(2\pi st) - \sin(2\pi Rt-\delta) \sin(2\pi st)$, we get 
\begin{align*}
g(t) &= \cos(2\pi Rt-\delta) \int_0^1 h(s)\, \cos(2\pi st)\, ds 
     - \sin(2\pi Rt-\delta) \int_0^1 h(s)\, \sin(2\pi st)\, ds \\
&= \frac{e^{i(2\pi Rt-\delta)}}{2} 
    \left( \int_0^1 h(s)\, \cos(2\pi st)\, ds 
    + i \int_0^1 h(s)\, \sin(2\pi st)\, ds \right) \\
&\quad +  \frac{e^{-i(2\pi Rt-\delta)}}{2}
    \left( \int_0^1 h(s)\, \cos(2\pi st)\, ds 
    - i \int_0^1 h(s)\, \sin(2\pi st)\, ds \right)\\
    &= \frac{e^{i(2\pi Rt-\delta)}}{2}  \int_0^1 h(s)\, e^{2\pi ist}\, ds + \frac{e^{-i(2\pi Rt-\delta)}}{2} \int_0^1 h(s)\, e^{-2\pi ist}\, ds.
\end{align*}
Define $H(t) := \int_0^1 h(s) \, e^{2\pi ist}\, ds$, so that
$$g(t) = \frac{e^{i(2\pi Rt-\delta)}}{2}  H(t) + \frac{e^{-i(2\pi  Rt-\delta)}}{2} \overline{H(t)}.$$

Note that $H(t) = \widehat{h}\left(-t\right)$ and $\widehat{H}(\xi) = h(\xi)$, implying $\operatorname{supp} \widehat{H} \subset [0, 1]$. As $\widehat{\overline{H}}(\xi) = \overline{\widehat{H}}(-\xi)$, we also have $\operatorname{supp} \widehat{\overline{H}} \subset [-1, 0]$. Using Theorem \ref{multi-interval-pls} \cite{k1} with $b = 1$, we get $$\|g\|_{L^2} \lesssim \|g\|_{L^2(E)},$$
for $R \ge 2$. Thus, $J_{1,\ep} \le \|g\|^2_{L^2} \lesssim \|g\|_{L^2(E)}^2.$ Finally, we would like to show $\|g\|_{L^2(E)}^2 \lesssim \|f\|^2_{L^2_\alpha(E)}$. 

Define $\widetilde{j_\alpha}(t) := A_\alpha t^{-\alpha-1/2} \cos(t-\delta)$. Note that 
\begin{align*}
    \|f\|^2_{L^2_\alpha(E)} &\approx \int_E |f(t)|^2\, t^{2\alpha + 1}\, dt\\
    &\approx \int_E \left|\int_0^1 \mathcal{F}_\alpha(f)(R+s)\,  j_\alpha(2\pi (R+s)t)\, (R+s)^{2\alpha + 1} ds\right|^2 \, t^{2\alpha + 1}\, dt,
\end{align*}
and
\begin{align*}
    \|g\|^2_{L^2(E)} &= \int_E \left|\int_0^1 \mathcal{F}_\alpha(f)(R+s) (R+s)^{\alpha+1/2} \cos(2\pi (R+s)t-\delta)\, ds\right|^2 \, dt\\
    &= \int_E \left|\int_0^1 \mathcal{F}_\alpha(f)(R+s) ((R+s)t)^{-\alpha-1/2} \cos(2\pi (R+s)t-\delta)\, (R+s)^{2\alpha + 1} ds\right|^2 \, t^{2\alpha + 1} dt\\
    &\approx \int_E \left|\int_0^1 \mathcal{F}_\alpha(f)(R+s) \, \widetilde{j}_\alpha(2\pi (R+s)t)\, (R+s)^{2\alpha + 1} ds\right|^2 \, t^{2\alpha + 1} dt.
\end{align*}

Once again, we can cut out $[0, \ep]$, since
\begin{align*}
    \|g\|^2_{L^2(E \cap [0,\ep])} &= \int_{E \cap [0, \ep]} \left|\int_0^1 \mathcal{F}_\alpha(f)(R+s) (R+s)^{\alpha+1/2} \cos(2\pi (R+s)t-\delta)\, ds\right|^2 \, dt\\
    &\le \int_{E \cap [0, \ep]} \left(\int_0^1 |\mathcal{F}_\alpha(f)(R+s)| (R+s)^{\alpha+1/2}\, ds\right)^2 \, dt\\
    &\le |E \cap [0,\ep]| \int_0^1 |\mathcal{F}_\alpha(f)(R+s)|^2 (R+s)^{2\alpha+1}\, ds \lesssim \ep \, \|f\|^2_{L^2_\alpha}.
\end{align*}
For the other part, we have
\begin{align*}
    \|g\|_{L^2(E \cap (\ep,\infty))} &\approx \left(\int_{E \cap (\ep,\infty)} \left|\int_0^1 \mathcal{F}_\alpha(f)(R+s) \, \widetilde{j}_\alpha(2\pi (R+s)t)\, (R+s)^{2\alpha + 1} ds\right|^2 \, t^{2\alpha + 1} dt\right)^{1/2}\\
    &\lesssim \clubsuit_1^{1/2} + \clubsuit_2^{1/2}
\end{align*}
where 
$$\clubsuit_1 := \int_{E \cap (\ep,\infty)} \left|\int_0^1 \mathcal{F}_\alpha(f)(R+s)\,  (\widetilde{j}_\alpha(2\pi(R+s)t) - {j}_\alpha(2\pi(R+s)t))\, (R+s)^{2\alpha + 1} ds\right|^2 \, t^{2\alpha + 1}\, dt,$$
and
$$\clubsuit_2 := \int_{E \cap (\ep,\infty)} \left|\int_0^1 \mathcal{F}_\alpha(f)(R+s)\,  {j}_\alpha(2\pi(R+s)t)\, (R+s)^{2\alpha + 1} ds\right|^2 \, t^{2\alpha + 1}\, dt.$$

We estimate $\clubsuit_2$ trivially, using $\clubsuit_2 \approx  \|f\|^2_{L^2_\alpha(E \cap (\ep, \infty))} \le \|f\|^2_{L^2_\alpha(E)},$ and $\clubsuit_1$ as follows.
\begin{align*}
    \clubsuit_1 &\lesssim \int_{E \cap (\ep,\infty)} \left(\int_0^1 |\mathcal{F}_\alpha(f)(R+s)|\,  ((R+s)t)^{-\alpha-3/2}\, (R+s)^{2\alpha + 1} ds\right)^2 \, t^{2\alpha + 1}\, dt\\
    &\lesssim \int_{E \cap (\ep,\infty)} \left(\int_0^1 |\mathcal{F}_\alpha(f)(R+s)|\, (R+s)^{\alpha + 1/2} (R+s)^{-1}\, ds\right)^2 t^{-2}\, dt\\
    &\lesssim \int_{E \cap (\ep,\infty)} \frac{dt}{t^2} \, \left(\int_0^1|\F_\alpha(f)(R+s)|^2 (R+s)^{2\alpha + 1}\, ds\right) \left(\int_0^1 \frac{ds}{(R+s)^2}\right)^2\\
    &\lesssim \frac{\|f\|^2_{L^2_\alpha}}{R^2} \int_{E \cap (\ep,\infty)} \frac{dt}{t^2} \le \frac{\|f\|^2_{L^2_\alpha}}{\ep R^2}.
\end{align*}
Combining these estimates, we get
\begin{align*}
    \|g\|^2_{L^2(E)} \lesssim \ep\, \|f\|^2_{L^2_\alpha} + \left(\frac{\|f\|_{L^2_\alpha}}{\sqrt{\ep} R} + \|f\|_{L^2_\alpha(E)}\right)^2 \le \|f\|^2_{L^2_\alpha} \left(\ep + \frac{1}{\ep R^2} \right) + \|f\|^2_{L^2_\alpha(E)}.
\end{align*}
As $\|f\|^2_{L^2_\alpha} \lesssim I_\ep + J_{1,\ep} + J_{2, \ep}$, we have
$$\|f\|^2_{L^2_\alpha} \lesssim \|f\|^2_{L^2_\alpha}\left(\ep +  \frac{1}{\ep R^2}\right) + \|f\|^2_{L^2_\alpha(E)}.$$

Now select $\ep = \frac{1}{R}$, then
$$\|f\|^2_{L^2_\alpha} \lesssim \frac{1}{R}\|f\|^2_{L^2_\alpha} + \|f\|^2_{L^2_\alpha(E)},$$
for all $R \ge 1$. Then, for large $R$, we get $\|f\|^2_{L^2_\alpha} \lesssim \|f\|^2_{L^2_\alpha(E)}$ with constants independent of $R$, as desired. More precisely, suppose $C > 0$ is such that $\|f\|^2_{L^2_\alpha} \le  C(\|f\|^2_{L^2_\alpha} R^{-1} + \|f\|^2_{L^2_\alpha(E)})$.  Then, for all $R \ge 2C$, 
$$\|f\|^2_{L^2_\alpha}  \le 2C \|f\|^2_{L^2_\alpha(E)} \lesssim \|f\|^2_{L^2_\alpha(E)}.$$
The theorem is proved.
\end{proof}

\section{Proof of Theorem \ref{thm2}}

In this section, $\alpha = m+ 1/2$ where $m$ is a non-negative integer. The proof is inspired by related work of Kovrizhkin \cite{k1}, and a part of the proof of Theorem \ref{thm2} relies on Nazarov's Turan inequality as stated in Lemma 3 of \cite{k1}. We record it below for convenience.

\begin{lemma}
\label{nazarov-turan-kovrizhkin}
    If $r(x) = \sum_{k=1}^N p_k(x) e^{2\pi i\lambda_k x}$, where $p_k(x)$ is a polynomial of degree $\le M-1$ and $E \subset I$ is measurable with $|E| > 0$, then
    \begin{equation}
    \label{nazarov-turan-kovrizhkin-eqn}
        \|r\|_{L^p(I)} \le \left(\frac{C|I|}{|E|}\right)^{NM - \frac{p-1}{p}}\, \|r\|_{L^p(E)}.
    \end{equation}
\end{lemma}

We also use the following Bernstein inequality of Ghobber and Jaming \cite{gj} to tackle the low-frequency block:

\begin{lemma}
    \label{lemma-bernstein}
    Let $E \subset \R^+$ be a $\mu_\alpha$-relatively dense set, i.e., for some $\gamma > 0$, 
    $$\mu_\alpha(E \cap [r,r+1]) \ge \gamma \mu_\alpha([r,r+1]),$$
    for all $r\ge 0$. Let $f\in L^2_\alpha(\R^+)$ such that $\supp \F_\alpha (f) \subset [0,R]$, and $g(x) := f(\sqrt x)$. Then,
    \begin{equation}
        \int_0^\infty |g^{(k)}(s)|^2\, s^{\alpha + k}\, ds \le (\pi R)^{2k} \int_0^\infty |g(s)|^2\, s^\alpha \, ds.
    \end{equation}
\end{lemma}

\begin{proof}[Proof of Theorem \ref{thm2}]
    
Assume $I_k := [\lambda_k, \lambda_k + 1]$. Without loss of generality, we may assume $\lambda_1 \in [0, 1)$ (otherwise, the proof that follows can be significantly simplified). We use Fourier-Bessel inversion to get
$$f(t) = c_\alpha\int_0^\infty \mathcal{F}_\alpha(f)(y) j_\alpha(2\pi ty)\, y^{2\alpha + 1}\, dy =  c_\alpha\,\sum_{k=1}^N \int_{I_k} \mathcal{F}_\alpha(f)(y) j_\alpha(2\pi ty)\, y^{2\alpha + 1}\, dy,$$
and
$$\|f\|^2_{L^2_\alpha} = c_\alpha^3 \int_0^\infty \left| \int_0^\infty \mathcal{F}_\alpha(f)(y) j_\alpha(2\pi ty)\, y^{2\alpha + 1}\, dy \right|^2\, t^{2\alpha + 1}\, dt.$$

First, we rewrite the Bessel kernel using Cauchy's integral theorem. Using the Poisson representation, 
$$j_\alpha(s) = \frac{\Gamma(\alpha + 1)}{\Gamma\left(\alpha + \frac12\right) \Gamma\left(\frac12\right)} \int_{-1}^1 e^{is x}\, (1-x^2)^{\alpha - 1/2}\, dx.$$
A classical contour integration argument \citepre{Appendix A}{ss-complex} yields
$$j_\alpha(s) = \frac{\Gamma(\alpha + 1)}{\Gamma\left(\alpha + \frac12\right) \Gamma\left(\frac12\right)} (-I_{+}(s) - I_{-}(s)),$$
where $$I_+(s) = ie^{is} \int_0^\infty e^{-sy} (1 - (1 + iy)^2)^{\alpha - 1/2}\, dy,$$
and
$$I_-(s) = -ie^{-is} \int_0^\infty e^{-sy} (1 - (-1 + iy)^2)^{\alpha - 1/2}\, dy.$$
Note that $(1 - (1+iy)^2)^m = (y^2 - 2iy)^m = \sum_{j=m}^{2m} a_j y^j$, and $(1 - (-1+iy)^2)^m = (y^2 + 2iy)^m = \sum_{j=m}^{2m} b_j y^j$ for appropriate constants $a_j, b_j\ \in \C$. This gives,
\begin{equation*}
    I_+(s) = ie^{is} \int_0^\infty e^{-sy} \, \sum_{j=m}^{2m} a_j y^j \, dy = ie^{is} \sum_{j=m}^{2m} j!\, a_j s^{-j-1},
\end{equation*}
and similarly,
\begin{equation*}
    I_-(s) = -ie^{-is} \sum_{j=m}^{2m} j!\, b_j s^{-j-1}.
\end{equation*}
Therefore,
\begin{equation*}
    j_\alpha(s) = \sum_{\pm } \sum_{j=m}^{2m} e^{\pm is} c_{\pm, j}\, s^{-j-1},
\end{equation*}
for appropriate constants $c_{\pm, j}\in \C$.\footnote{This finite expansion for $j_\alpha(s)$ can also be obtained in a more direct way using Rayleigh's formula, where $\alpha = m+1/2$ and $m$ is a non-negative integer: $$j_{m+1/2}(x) = (2m+1)!!\, (-1)^m \left(\frac{1}{x} \frac{d}{dx} \right)^m \frac{\sin x}{x}$$} Substituting $s = 2\pi ty$, we have 
\begin{equation*}
    j_\alpha(2\pi ty) = \sum_{\pm } \sum_{j=m}^{2m} e^{\pm 2\pi i ty} c_{\pm, j}\, (2\pi t)^{-j-1} y^{-j-1},
\end{equation*}
giving
\begin{align*}
    f(t) &= c_\alpha \sum_{k=1}^N \int_{I_k} \F_\alpha(f)(y) j_\alpha(2\pi ty)\,  y^{2\alpha +1}\, dy\\
    &= c_\alpha\sum_{\pm} \sum_{k=1}^N \sum_{j=m}^{2m} c_{\pm, j}\, (2\pi t)^{-j-1} \int_{I_k} \F_\alpha f(y)\,  e^{\pm 2\pi i ty}\,  y^{2\alpha - j}\, dy\\
    &= c_\alpha \sum_{\pm} \sum_{k=1}^N \sum_{j=m}^{2m} c_{\pm, j}\, (2\pi t)^{-j-1} \int_{0}^1 \F_\alpha f(\lambda_k + s)\,  e^{\pm 2\pi i t(\lambda_k + s)}\,  (\lambda_k + s)^{2\alpha - j}\, ds. 
\end{align*}
For convenience, define $H_{j,k}(s) := \F_\alpha (f)(\lambda_k + s) (\lambda_k + s)^{2\alpha - j}$, $g_{\pm, j, k}(t) := \int_0^1 H_{j,k}(s)\, e^{\pm 2\pi i ts}\, ds$, and $G_{\pm, j, k}(t) := g_{\pm, j, k}(t) e^{\pm 2\pi i\lambda_k t}$. Then,
\begin{align*}
    f(t) &= c_\alpha \sum_{\pm} \sum_{k=1}^N \sum_{j=m}^{2m} c_{\pm, j}\, (2\pi t)^{-j-1} \int_{0}^1 H_{j,k}(s)\,  e^{\pm 2\pi i t(\lambda_k + s)}\, ds\\
    &= c_\alpha \sum_{\pm} \sum_{k=1}^N \sum_{j=m}^{2m} c_{\pm, j}\, (2\pi t)^{-j-1}\,  g_{\pm, j,k}(t) \, e^{\pm 2\pi i\lambda_k t}\\
    &= c_\alpha \sum_{\pm} \sum_{k=1}^N \sum_{j=m}^{2m} c_{\pm, j}\, (2\pi t)^{-j-1} \, G_{\pm, j,k}(t) = c_\alpha \sum_{k=1}^N F_k(t),
\end{align*}
where $F_k(t) := \sum_{\pm} \sum_{j=m}^{2m} c_{\pm, j}\, (2\pi t)^{-j-1} \, G_{\pm, j,k}(t)$. Taking the Fourier transform of $g_{\pm,j,k}$, we see that
$$\widehat{g_{+,j,k}}(\xi) = \begin{cases} H_{j,k}(\xi) & 0\le \xi \le 1\\ 0 & \text{otherwise}, \end{cases}$$
and
$$\widehat{g_{-,j,k}}(\xi) = \begin{cases} H_{j,k}(-\xi) & -1\le \xi \le 0\\ 0 & \text{otherwise}. \end{cases}$$
So, $\operatorname{supp} \widehat{g_{+,j,k}} \subset [0,1]$ and $\operatorname{supp} \widehat{g_{-,j,k}} \subset [-1, 0]$. Since $G_{\pm, j, k}(t) = g_{\pm, j, k}(t) e^{\pm 2\pi i\lambda_k t}$, we have $\widehat{G_{\pm, j, k}}(\xi) = \widehat{g_{\pm, j, k}}(\xi \mp \lambda_k)$, giving $\operatorname{supp} \widehat{G_{+, j, k}} \subset [\lambda_k, \lambda_k + 1]$ and $\operatorname{supp} \widehat{G_{-, j, k}} \subset [-(\lambda_k +1 ), -\lambda_k]$.

As in the proof of Theorem \ref{thm1}, we first cut out $[0, \ep)$. We have
\begin{align*}
    \|f\|_{L^2_\alpha([0, \ep)])}^2 &= c_\alpha \int_0^\ep |f(t)|^2\, t^{2\alpha + 1}\, dt\\
    &\le c_\alpha^3 \int_0^\ep \left( \int_0^\infty |\mathcal{F}_\alpha(f)(y)| |j_\alpha(2\pi ty)|\, y^{2\alpha + 1}\, dy \right)^2\, t^{2\alpha + 1}\, dt\\
    &\lesssim \int_0^\ep \left( \int_0^\infty |\mathcal{F}_\alpha(f)(y)|\, y^{\alpha + 1/2} \, dy \right)^2 \, dt\\
    &\lesssim \ep N\, \|\F_\alpha f\|^2_{L^2_\alpha} \lesssim \ep \, \|f\|^2_{L^2_\alpha},
\end{align*}
where (\ref{bessel-function-estimate}) is used in the third line.

Next, for $[\ep,\infty)$, we mimic the Taylor approximation idea from \cite{k1}. Write $[\ep, \infty) = \bigcup_{n=0}^\infty [\ep + n, \ep + n + 1] =: \bigcup_{n=0}^\infty A_n$, and set $\widetilde{A_n} := [(\ep + n)^2, (\ep + n + 1)^2]$. The form of the approximation differs depending on whether $k = 1$ or $k \ge 2$, due to the form of the Bernstein inequality that we subsequently use.

For $k = 1$, define $\Phi(u) := F_1(\sqrt{u})$ and let $r_{1,n}$ be the $M-1$-degree Taylor polynomial of $\Phi$, centered at $u = (\ep + n)^2$. Let $p_{1,n}(t) := r_{1,n}(t^2)$ for $t \in A_n$. 

When $k \ge 2$, let $w_j(t) := (2\pi t)^{-j-1}$, and $h_{\pm, j,k}(t) := w_j(t) g_{\pm, j,k}(t)$. Let $p_{j,k,n}(t)$ be the $M-1$-degree Taylor polynomial of $g_{\pm,j,k}(t)$, centered at $t = \ep + n$, i.e., the left-endpoint of $A_n$.

Define $$r_n(t) := p_{1,n}(t) + \sum_{\pm} \sum_{k=2}^N \sum_{j=m}^{2m} c_{\pm, j}\, w_j(t)\,  p_{j,k,n}(t)\, e^{\pm 2\pi i\lambda_k t},$$
and 
$$T_n(t) :=  \underbrace{F_1(t) - p_{1,n}(t)}_{T_n^{(1)}(t)} + T_n^{(2)}(t),$$
where
$$T_n^{(2)}(t) := \sum_{\pm} \sum_{k=2}^N \sum_{j=m}^{2m} c_{\pm, j}\, w_j(t)\left(g_{\pm, j,k}(t)  - p_{j,k,n}(t) \right)\, e^{\pm 2\pi i\lambda_k t}.$$

We can control $r_n$ on $A_n$ using Nazarov's Turan inequality, i.e., Lemma \ref{nazarov-turan-kovrizhkin}. First, define $\widetilde{r_n}(t) := (2\pi t)^{2m+1}\, r_n(t)$, so
$$\widetilde{r_n}(t) = (2\pi t)^{2m+1}\,p_{1,n}(t) + \sum_{\pm} \sum_{k=2}^N \sum_{j=m}^{2m} c_{\pm, j}\,(2\pi t)^{2m-j}\,  p_{j,k,n}(t)\, e^{\pm 2\pi i\lambda_k t},$$
is a suitable candidate for Lemma \ref{nazarov-turan-kovrizhkin}. We get
\begin{equation}
\label{remez-1}
    \|\widetilde{r_n}\|_{L^2(A_n)} \le \left(\frac{C|A_n|}{|E\cap A_n|} \right)^{(2N-1)(2M + 2m - 1) - \frac{1}{2}}\, \|\widetilde{r_n}\|_{L^2(E\cap A_n)},
\end{equation}
as there are $2(N-1) + 1 = 2N-1$ polynomials, and the degree of each polynomial is bounded by $\max\{M + m-1, 2M-2 + 2m+1\} = 2M + 2m - 1$. Since the weight $(2\pi t)^{2m+1}$ is almost constant on $A_n$, (\ref{remez-1}) can be converted to an  estimate for $r_n$. In particular,
\begin{equation}
\label{remez-2}
    \|{r_n}\|_{L^2(A_n)} \le 2^{2m+1}\, \left(\frac{C|A_n|}{|E\cap A_n|} \right)^{(2N-1)(2M + 2m - 1) - \frac{1}{2}}\, \|{r_n}\|_{L^2(E\cap A_n)},
\end{equation}
for all $n\ge 1$, as 
$$\frac{\max_{A_n}\, (2\pi t)^{2m+1}}{\min_{A_n} (2\pi t)^{2m+1}}  = \left(1 + \frac{1}{\ep + n} \right)^{2m+1} \le 2^{2m+1},$$
for all $n \ge 1$. If $n = 0$, we have
\begin{equation}
\label{remez-3}
    \|{r_0}\|_{L^2(J_0)} \le \left(1 + \frac{1}{\ep} \right)^{2m+1}\, \left(\frac{C|J_0|}{|E\cap J_0|} \right)^{(2N-1)(2M + 2m - 1) - \frac{1}{2}}\, \|{r_0}\|_{L^2(E\cap J_0)}.
\end{equation}
By Proposition \ref{relative-density-equivalence}, there exists $\widetilde{\gamma}$ such that $|E \cap A_n| \ge \widetilde{\gamma}\, |A_n|$ for each $n \ge 0$. Using this relative density condition, and summing up, we obtain
$$\sum_{n=0}^\infty \|r_n\|^2_{L^2(A_n)} \le (1 + \ep^{-1})^{4m+2}\left(C\,\widetilde\gamma^{-1} \right)^{(4N-2)(2M + 2m - 1) - 1} \sum_{n=0}^\infty  \|{r_n}\|_{L^2(E\cap A_n)}^2,$$
($\ep < 1$ ensures $\max \{2^{4m+2}, (1 + \ep^{-1})^{4m+2}\} = (1 + \ep^{-1})^{4m+2}$). Similarly, we can lift this estimate to $L^2_\alpha$ spaces, with a cost of a $(1 + \ep^{-1})^{2m+2}$ factor, i.e., 
$$\sum_{n=0}^\infty \|r_n\|^2_{L^2_\alpha(A_n)} \le (1 + \ep^{-1})^{6m+4}\left(C\,\widetilde\gamma^{-1} \right)^{(4N-2)(2M + 2m - 1) - 1} \sum_{n=0}^\infty  \|{r_n}\|_{L^2_\alpha(E\cap A_n)}^2$$

Now, we study the remainder terms. We analyze the low-frequency block first, i.e., $F_1$. 

Observe that 
\begin{align*}
    \|T_n^{(1)}\|^2_{L^2_\alpha(A_n)} \approx \|\Phi - r_{1,n}\|_{L^2(\widetilde{A_n}; \,u^\alpha du)}^2
\end{align*}
since the substitution  $u = t^2$ leads to $2t^{2\alpha + 1}\, dt = u^\alpha\, du$, and so
\begin{align*}
    \|T_n^{(1)}\|^2_{L^2_\alpha(A_n)} \approx \int_{A_n} |F_1(t) - p_{1,n}(t)|^2\, t^{2\alpha + 1}\, dt \approx \int_{\widetilde{A_n}} |\Phi(u) - r_{1,n}(u)|^2\, u^{\alpha}\, du. 
\end{align*} 

Using Taylor's remainder theorem,
$$\Phi(u) - r_{1,n}(u) = \frac{1}{(M-1)!} \int_{(\ep + n)^2}^u (u-s)^{M-1}\, \Phi^{(M)}(s)\, ds,$$
for every $u \in \widetilde{A_n}$.

In preparation for Schur's test, we define an operator $S_n: L^2(\widetilde{A_n}; s^{\alpha+ M} ds) \to L^2(\widetilde{A_n}; u^\alpha\, du)$ between weighted spaces by
$$(S_n h)(u) = \frac{1}{(M-1)!} \int_{\widetilde{A_n}} (u-s)_+^{M-1}\, h(s)\, ds,$$
where $x_+ := \max\{x,0\}$. We introduce two canonical isometries $\mathcal{U}, \mathcal{V}$ to define another operator $\widetilde{S_n} := \mathcal{V}\inv S_n\, \mathcal{U}\inv$ between unweighted $L^2$ spaces, which is our candidate for Schur's test. Let 
$\mathcal{U}: L^2(\widetilde{A_n}; s^{\alpha + M} ds) \to L^2(\widetilde{A_n}; ds)$ be given by $$(\mathcal{U}h)(s) := s^{\frac{\alpha + M}{2}}\, h(s),$$
and $\mathcal{V}: L^2(\widetilde{A_n}; du) \to L^2(\widetilde{A_n}; u^\alpha du)$ be given by
$$(\mathcal{V}h)(u) := u^{-\alpha/2}\, h(u).$$
Then,
$$(\widetilde{S_n} h)(u) = \int_{\widetilde{A_n}} K_n(u,s)\, h(s)\, ds,$$
where 
$$K_n(u,s) := \frac{1}{(M-1)!}\, (u-s)_+^{M-1}\, u^{\alpha/2}\, s^{-\frac{(\alpha + M)}{2}}.$$
Therefore,
$$\|\Phi - r_{1,n}\|_{L^2(\widetilde{A_n}; \,u^\alpha du)}^2 = \|\widetilde{S_n}(s^{\frac{\alpha + M}{2}}\, \Phi^{(M)}(s))\|_{L^2(\widetilde{A_n};\, du)}^2 \le \|\widetilde{S_n}\|^2 \|\Phi^{(M)}\|_{L^2(\widetilde{A_n}; \, s^{\alpha + M} ds)}^2.$$
We obtain bounds for $\|\widetilde{S_n}\|$ using Schur's test on the kernel $K_n(u,s)$. Let $b_n := \ep + n$. Then, for $u,s \in \widetilde{A_n}$, we have $s \ge b_n^2$, $u \le (b_n + 1)^2$, and $|u-s| \le |\widetilde{A_n}| = 2b_n + 1$. This gives
$$u^{\alpha/2}\, s^{-\frac{(\alpha + M)}{2}} = \left(\frac{u}{s} \right)^{\alpha/2} s^{-M/2} \le \left(1 + \frac{1}{b_n}\right)^\alpha\, b_n^{-M}.$$
Hence, $$K_n(u,s) \le  \frac{1}{(M-1)!}\, (u-s)_+^{M-1} \left(1 + \frac{1}{b_n}\right)^\alpha\, b_n^{-M},$$ and therefore,
$$\int_{\widetilde{A_n}} K_n(u,s)\, ds \le \frac{1}{(M-1)!}\, \left(1 + \frac{1}{b_n}\right)^\alpha\, b_n^{-M} \int_0^{2b_n+1} x^{M-1}\, dx = \frac{1}{M!}\, \left(1 + \frac{1}{b_n}\right)^\alpha \left(2 + \frac{1}{b_n} \right)^M.$$
It follows that
    $$\sup_{u\in \widetilde{A_n}}\int_{\widetilde{A_n}} K_n(u,s)\, ds \le \frac{1}{M!} \left(1 + \frac{1}{b_n}\right)^\alpha\, \left(2 + \frac{1}{b_n}\right)^M,$$
and a similar calculation gives
     $$\sup_{s\in \widetilde{A_n}}\int_{\widetilde{A_n}} K_n(u,s)\, du \le \frac{1}{M!} \left(1 + \frac{1}{b_n}\right)^\alpha\, \left(2 + \frac{1}{b_n}\right)^M.$$
By Schur's test,
$$\|\widetilde{S_n}\|_{L^2\to L^2} \le \frac{1}{M!} \left(1 + \frac{1}{b_n}\right)^\alpha\, \left(2 + \frac{1}{b_n}\right)^M \le \frac{2^\alpha 3^M}{M!},$$
for all $n\ge 1$, and
$$\|\widetilde{S_0}\|_{L^2\to L^2} \le \frac{(1 + \ep^{-1})^\alpha (2 + \ep^{-1})^M}{M!},$$
if $n = 0$.

Next, using Lemma \ref{lemma-bernstein}, we have
\begin{equation}
    \label{lemma-bernstein-2}
    \|\Phi^{(M)}\|_{L^2(\R^+; \, s^{\alpha + M} ds)}^2 \lesssim (2\pi)^{2M} \int_0^\infty |\Phi(s)|^2\, s^\alpha\, ds \lesssim 2^{6M}\,\|F_1\|^2_{L^2_\alpha},
\end{equation}
since $\operatorname{supp} \F_\alpha(F_1) \subset [\lambda_1, \lambda_1 + 1] \subset [0,2]$. Substituting these estimates into
$$\|T_n^{(1)}\|^2_{L^2_\alpha(A_n)}  \lesssim \|\widetilde{S_n}\|^2 \|\Phi^{(M)}\|_{L^2(\widetilde{A_n}; \, s^{\alpha + M} ds)}^2,$$
summing up, and using $\ep < 1$, we obtain
$$\sum_{n=0}^\infty \|T_n^{(1)}\|^2_{L^2_\alpha(A_n)} \lesssim \frac{(1 + \ep^{-1})^{2\alpha} (2 + \ep^{-1})^{2M}}{(M!)^2} \sum_{n=0}^\infty \|\Phi^{(M)}\|_{L^2(\widetilde{A_n}; \, s^{\alpha + M} ds)}^2.$$
This leads to the following remainder estimate for the low-frequency block:
$$\sum_{n=0}^\infty \|T_n^{(1)}\|^2_{L^2_\alpha(A_n)} \lesssim \frac{2^{6M}(1 + \ep^{-1})^{2\alpha} (2 + \ep^{-1})^{2M}}{(M!)^2}\|F_1\|^2_{L^2_\alpha}.$$

Next, we control the remainder term for $k \ge 2$, i.e., we treat the high frequency block $T_n^{(2)}(t)$. Note that
$$|T_n^{(2)}(t)| \lesssim \sum_{\pm} \sum_{k=2}^N \sum_{j=m}^{2m} |w_j(t)| \left|g_{\pm, j,k}(t)  - p_{j,k,n}(t) \right|,$$
which gives
$$\|T_n^{(2)}\|^2_{L^2_\alpha(A_n)} \lesssim \sum_{\pm} \sum_{k=2}^N \sum_{j=m}^{2m} \|w_j \left(g_{\pm, j,k} - p_{j,k,n} \right)\|^2_{L^2_\alpha(A_n)}.$$

By Taylor's remainder theorem, 
$$g_{\pm, j,k}(t)  - p_{j,k,n}(t) = \frac{1}{(M-1)!} \int_{\ep + n}^{t} g^{(M)}_{\pm,j,k}(s) (t-s)^{M-1}\, ds.$$

Define $(S_{j,n} h)(t) := \int_{A_n} K_{j,n}(t,s)\, h(s)\, ds,$
where 
$$K_{j,n}(t,s) = \frac{1}{(M-1)!} (t-s)^{M-1}_+\, w_j(t)\, t^{\alpha + \frac{1}{2}}.$$
Clearly, we want bounds for $\|S_{j,n}\|$, as 
$$\|w_j \left(g_{\pm, j,k} - p_{j,k,n} \right)\|_{L^2_\alpha(A_n)} \approx \|S_{j,n}(g_{\pm,j,k}^{(M)})\|_{L^2(A_n;\, dt)} \le \|S_{j,n}\| \|g_{\pm,j,k}^{(M)}\|_{L^2(A_n;\, ds)}.$$
We have 
$$\sup_{t\in A_n}\int_{A_n} K_{j,n}(t,s)\, ds \le \frac{(2\pi)^{-j-1}}{M!}\, \sup_{t\in A_n} t^{m-j},$$
and 
$$\sup_{s\in A_n}\int_{A_n} K_{j,n}(t,s)\, dt \le \frac{(2\pi)^{-j-1}}{M!}\, \sup_{t\in A_n} t^{m-j},$$
giving 
$$\|S_{j,n}\|_{L^2\to L^2} \le \frac{(2\pi)^{-j-1}}{M!}\, \sup_{t\in A_n} t^{m-j},$$
and therefore
$$\|S_{j,n}\|_{L^2\to L^2} \le \begin{cases}
    \frac{(2\pi)^{-j-1}}{M!} & \text{if }n\ge 1,\\
    \frac{(2\pi)^{-j-1}\, \ep^{m-j}}{M!} & \text{if }n=0.
\end{cases}$$

As earlier, summing up, and using $\ep < 1$ with $(2\pi)^{-2j-2} \le 1$, 
$$\sum_{n=0}^\infty \|w_j \left(g_{\pm, j,k} - p_{j,k,n} \right)\|_{L^2_\alpha(A_n)}^2 \lesssim \frac{\ep^{-2m}}{(M!)^2} \, \|g_{\pm,j,k}^{(M)}\|_{L^2(\R)}^2.$$
As $\ep < 1$ and $j \le 2m$, we have $\ep^{2m-2j} \le \ep^{-2m}$. Finally, as $\operatorname{supp} g_{\pm,j,k} \subset [-1,1]$, 
$$\|g_{\pm,j,k}^{(M)}\|_{L^2(\R)}^2 = \|\widehat{g_{\pm,j,k}^{(M)}}\|_{L^2(\R)}^2 \le (2\pi)^{2M} \|\widehat{g_{\pm,j,k}}\|^2_{L^2(\R)} = (2\pi)^{2M} \|g_{\pm,j,k}\|^2_{L^2(\R)}.$$
Summing over the $(\pm, j, k)$ indices, we obtain the following remainder estimate for the high-frequency block:
$$\sum_{n=0}^\infty \|T_n^{(2)}\|^2_{L^2_\alpha(A_n)} \lesssim \frac{(2\pi)^{2M}\, \ep^{-2m}}{(M!)^2} \sum_{\pm} \sum_{k=2}^N \sum_{j=m}^{2m} \|g_{\pm,j,k}\|^2_{L^2(\R)}.$$
As $\lambda_k \ge 1$ for $k \ge 2$, we have
\begin{align*}
    \|\widehat{g_{\pm,j,k}}\|_{L^2(\R)}^2 = \int_0^1 |H_{j,k}(\xi)|^2\, d\xi &= \int_0^1 |\F_\alpha(f)(\lambda_k + s)|^2 \, (\lambda_k + s)^{4m+2-2j}\, ds\\
    &\le \int_{\lambda_k}^{\lambda_k + 1} |\F_\alpha (f)(s)|^2\, s^{2m+2}\, ds = \int_{\lambda_k}^{\lambda_k + 1} |\F_\alpha (f)(s)|^2\, s^{2\alpha + 1}\, ds,
\end{align*}
and therefore the high-frequency estimate can be simplified to 
$$\sum_{n=0}^\infty \|T_n^{(2)}\|^2_{L^2_\alpha(A_n)} \lesssim \frac{(2\pi)^{2M}\, \ep^{-2m}}{(M!)^2} \sum_{k=2}^N \|F_k\|^2_{L^2_\alpha}.$$

Finally, we combine all estimates to show $\|f\|_{L^2_\alpha(\R^+)} \lesssim \|f\|_{L^2_\alpha(E)}$. For convenience, define the following quantities:
$$C_{\text{NT}} := (1 + \ep^{-1})^{6m+4}\left(C\,\widetilde\gamma^{-1} \right)^{(4N-2)(2M + 2m - 1) - 1},$$
$$C_{\text{low}} := \frac{2^{6M}(1 + \ep^{-1})^{2\alpha} (2 + \ep^{-1})^{2M}}{(M!)^2},$$
$$C_{\text{high}} := \frac{(2\pi)^{2M}\, \ep^{-2m}}{(M!)^2}.$$
Observe that $C_{\text{NT}}\,(C_{\text{low}} + C_{\text{high}}) \to 0$ as $M\to\infty$. Finally,
\begin{align*}
\|f\|_{L^2_\alpha(\R^+)}^2 &\lesssim \ep\, \|f\|_{L^2_\alpha(\R^+)}^2 + \|f\|^2_{L^2_\alpha([\ep, \infty])}\\
&\lesssim \ep\, \|f\|_{L^2_\alpha(\R^+)}^2 + \sum_{n=0}^\infty \|r_n\|^2_{L^2_\alpha(A_n)} + \sum_{n=0}^\infty \|T_n^{(1)}\|^2_{L^2_\alpha(A_n)} + \sum_{n=0}^\infty \|T_n^{(2)}\|^2_{L^2_\alpha(A_n)}\\
&\lesssim \ep\, \|f\|_{L^2_\alpha(\R^+)}^2 + C_{\text{NT}} \sum_{n=0}^\infty \|r_n\|^2_{L^2_\alpha(E\cap A_n)} + (C_{\text{low}} + C_{\text{high}}) \|f\|^2_{L^2_\alpha(\R^+)}\\
&\lesssim \ep\, \|f\|_{L^2_\alpha(\R^+)}^2 +  C_{\text{NT}}\, \|f\|^2_{L^2_\alpha(E)} + (C_{\text{NT}} + 1)(C_{\text{low}} + C_{\text{high}}) \|f\|^2_{L^2_\alpha(\R^+)}.
\end{align*}
Let $\beta := \beta(\alpha, N, \gamma) > 0$ be such that
$$\|f\|_{L^2_\alpha(\R^+)}^2  \le \ep\beta\, \|f\|_{L^2_\alpha(\R^+)}^2 +  \beta C_{\text{NT}}\, \|f\|^2_{L^2_\alpha(E)} + \beta (C_{\text{NT}} + 1)(C_{\text{low}} + C_{\text{high}}) \|f\|^2_{L^2_\alpha(\R^+)}.$$
Now, choose $\ep$ small enough so that $\ep \beta < \frac{1}{4}$. Once $\ep$ is fixed, we choose $M$ large enough to ensure $\beta (C_{\text{NT}} + 1)(C_{\text{low}} + C_{\text{high}}) < \frac{1}{4}$. In conclusion, we have 
$$\|f\|_{L^2_\alpha(\R^+)} \lesssim_{\alpha, N, \gamma} \|f\|_{L^2_\alpha(E)}.$$

\end{proof}

\section{An Extension of Theorem \ref{thm-damped}}

Theorem \ref{thm-damped} can be extended to functions that can be written as a product of a radial function and a spherical harmonic. The following calculation obtains the estimate (\ref{pls-independent}), using the estimate for radial functions (Theorem \ref{thm1}). 

    Suppose $f(x) = F(|x|) \,Y_k\left(\frac{x}{|x|} \right)$. By the Bochner identity,
\begin{equation}
\label{bochner-1}
    \widehat{f}(y) = \widehat{FY_k}(y) \approx Y_k\bigg(\frac{y}{|y|}\bigg)\, |y|^k \int_0^\infty F(r)\, j_{n/2 + k - 1}(2\pi r|y|)\, r^{n+k-1}\, dr.
\end{equation}
The right-side can be expressed as the Fourier-Bessel transform of $g(r) = r^{-k}\, F(r)$, of order $\alpha = n/2 + k -1$. If $\alpha = n/2 + k - 1$, then $2\alpha + 1 = n + 2k - 1$. Thus,
\begin{equation}
\label{bochner-2}
    \widehat{f}(y) \approx  Y_k\bigg(\frac{y}{|y|}\bigg)\, |y|^k \int_0^\infty g(r)\, j_{\alpha}(2\pi r|y|)\, r^{2\alpha +1}\, dr =  Y_k\bigg(\frac{y}{|y|}\bigg)\, |y|^k\, \F_\alpha(g)(|y|).
\end{equation}

Now, suppose $\operatorname{supp} \widehat{f} \subset \{\xi \in R^n: R \le |\xi| \le R+1\}$, and $E$ satisfies $\mu_\alpha$-relative density in each radial direction. In other words, if $E_\omega$ is the slice of $E$ along $\omega\in S^{n-1}$, then, omitting pushforwards, we ask $\mu_\alpha (E_\omega \cap [R\omega,(R+1)\omega]) \gtrsim \mu_\alpha([R\omega, (R+1)\omega])$ for all $R \ge 0$. By \Cref{bochner-2}, we have $\operatorname{supp} \F_\alpha(g) \subset [R,R+1]$. By Theorem \ref{thm1}, we get 
\begin{equation}
    \label{bochner-3}
    \|g\|_{L^2_\alpha(\R^+)} \lesssim \|g\|_{L^2_\alpha(E_\omega)},
\end{equation}
for every $\omega \in S^{n-1},$ where $\alpha = n/2 + k - 1$. Finally,
\begin{align*}
    \|f\|_{L^2(\R^n)}^2 &= \int_{0}^\infty \int_{S^{n-1}} |F(r)|^2\, |Y_k(\omega)|^2 \, r^{n-1}\, d\sigma(\omega) \, dr\\ 
    &= \int_{0}^\infty \int_{S^{n-1}} |g(r)|^2\, |Y_k(\omega)|^2 \, r^{n+2k-1}\, d\sigma(\omega) \, dr\\
    &= \|Y_k\|^2_{L^2(S^{n-1})}\, \|g\|_{L^2_\alpha(\R^+)}^2,
\end{align*}
and similarly, 
\begin{align*}
    \|f\|^2_{L^2(E)} &=  \int_{0}^\infty \int_{S^{n-1}} \mathbf{1}_{E}(r\omega)\, |g(r)|^2\, |Y_k(\omega)|^2 \, r^{n+2k-1}\, d\sigma(\omega) \, dr\\
    &= \int_{S^{n-1}} |Y_k(\omega)|^2 \int_{E_\omega} |g(r)|^2\, r^{n+2k-1}\, dr\, d\sigma(\omega)\\
    &= \int_{S^{n-1}} |Y_k(\omega)|^2\, \|g\|^{2}_{L^2_\alpha(E_\omega)}\, d\sigma(\omega).
\end{align*}
Starting with \Cref{bochner-3}, multiplying both sides by $|Y_k(\omega)|^2$ and integrating over $(S^{n-1}, d\sigma)$, we get $\|f\|_{L^2(\R^n)} \lesssim \|f\|_{L^2(E)}$ as desired.

\section*{Acknowledgments}
The authors thank Sergey Tikhonov for suggesting the extension of Theorem \ref{thm-damped} as presented in the final section of the paper.  We also thank Alexander Volberg for suggesting to consider the radial case of Question \ref{gccques}, which initiated the project.

Research supported in part by NSF grants DMS-2453251 and DMS-2049477.  This research was undertaken in part while B.J. was a Simons’ Fellow.

\nocite{*}
\bibliographystyle{plain}\bibliography{pls}

@article{green-jaye-mitkovski,
  author  = {Green, Walton and Jaye, Benjamin and Mitkovski, Mishko},
  title   = {{Uncertainty Principles Associated to Sets Satisfying the Geometric Control Condition}},
  journal = {Journal of Geometric Analysis},
  volume  = {32},
  pages   = {80},
  year    = {2022},
  doi     = {10.1007/s12220-021-00830-x},
}

@article{suzuki-inami,
  author  = {Inami, Kotaro and Suzuki, Soichiro},
  title   = {{Equivalence between the energy decay of fractional damped Klein–Gordon equations and geometric conditions for damping coefficients}},
  journal = {Proceedings of the American Mathematical Society Series B},
  volume  = {10},
  pages   = {422--430},
  year    = {2023},
  doi     = {10.1090/bproc/197},
}

@article{suzuki-damping,
  author  = {Suzuki, Soichiro},
  title   = {{The uncertainty principle and energy decay estimates of the fractional Klein-Gordon equation with space-dependent damping}},
  journal = {arXiv:2212.02481},
  year    = {2022},
  doi     = {10.48550/arXiv.2212.02481},
}

@article{gj,
  author    = {Ghobber, Saifallah and Jaming,  Philippe},
  title     = {{The Logvinenko--Sereda theorem for the Fourier--Bessel transform}},
  journal   = {Integral Transforms and Special Functions},
  volume    = {24},
  number    = {6},
  pages     = {470--484},
  year      = {2012},
  doi       = {10.1080/10652469.2012.708868},
}

@article{malhi2020energy,
  author  = {Malhi, Sachin and Stanislavova, Milena},
  title   = {{On the energy decay rates for the 1D damped fractional Klein–Gordon equation}},
  journal = {Mathematische Nachrichten},
  volume  = {293},
  pages   = {363--375},
  year    = {2020},
  doi     = {10.1002/mana.201800417},
}

@article{k1,
  author  = {Kovrijkine, Oleg},
  title   = {{Some Results Related to the Logvinenko-Sereda Theorem}},
  journal = {Proceedings of the American Mathematical Society},
  volume  = {129},
  number = {10},
  pages   = {3037--3047},
  year    = {2001},
  doi     = {},
}

@book{ss-complex,
  author  = {Shakarchi, Rami and Stein, Elias M.},
  title   = {Complex Analysis},
  publisher = {Princeton University Press},
  year    = {2003}
}

@book{muscschlag,
    author = {Muscalu, Camil and Schlag, Wilhelm},
    title = {Classical and multilinear harmonic analysis. Vol. I},
    series = {Cambridge Studies in Advanced Mathematics},
publisher = {Cambridge University Press},
    year = {2013}
}

@article{nazarov,
    author = {Nazarov, Fedor},
    title = {Local estimates for exponential polynomials and their applications to inequalities of the uncertainty principle type.},
    journal = {Algebra i Analiz},
volume  = {5},
  pages   = {3--66},
    year = 1993
}

@Article{b-l-r,
  title={Sharp sufficient conditions for the observation, control, and stabilization of waves from the boundary},
  author={Bardos, Claude and Lebeau, Gilles and Rauch, Jeffrey},
  journal={SIAM J. Control Optim.},
  volume={30},
  number={5},
  pages={1024--1065},
  year={1992}
}

@article{borichev10,
  title={Optimal polynomial decay of functions and operator semigroups},
  author={Borichev, Alexander and Tomilov, Yuri},
  journal={Mathematische Annalen},
  volume={347},
  number={2},
  pages={455--478},
  year={2010},
  publisher={Springer}
}

@article{burq16,
  title={Exponential decay for the damped wave equation in unbounded domains},
  author={Burq, Nicolas and Joly, Romain},
  journal={Communications in Contemporary Mathematics},
  volume={18},
  number={06},
  pages={1650012},
  year={2016},
  publisher={World Scientific}
}

@article{egidi18,
  title={Sharp geometric condition for null-controllability of the heat equation on $\mathbb{R}^d$ and consistent estimates on the control cost},
  author={Egidi, Michela and Veseli{\'c}, Ivan},
  journal={Archiv der Mathematik},
  volume={111},
  number={1},
  pages={85--99},
  year={2018},
  publisher={Springer}
}

@article{folland97,
  title={The uncertainty principle: a mathematical survey},
  author={Folland, Gerald B and Sitaram, Alladi},
  journal={Journal of Fourier analysis and applications},
  volume={3},
  number={3},
  pages={207--238},
  year={1997},
  publisher={Springer}
}

@article{gearhart78,
  title={Spectral theory for contraction semigroups on Hilbert space},
  author={Gearhart, Larry},
  journal={Transactions of the American Mathematical Society},
  volume={236},
  pages={385--394},
  year={1978}
}

@article{green2019decay,
  title={On the Energy Decay Rate of the Fractional Wave Equation on $\mathbb{R}$ with Relatively Dense Damping},
  author={Green, Walton},
  journal={arXiv preprint arXiv:1904.10946},
  year={2019}
}

@book{havin12,
  title={The uncertainty principle in harmonic analysis},
  author={Havin, Victor and J{\"o}ricke, Burglind},
  volume={28},
  year={2012},
  publisher={Springer Science \& Business Media}
}

@article{jaffard01,
  title={Estimates of the constants in generalized {Ingham}'s inequality and applications to the control of the wave equation},
  author={Jaffard, St{\'e}phane and Micu, Sorin},
  journal={Asymptotic Analysis},
  volume={28},
  number={3, 4},
  pages={181--214},
  year={2001},
  publisher={IOS Press}
}

@article{logvinenko74,
  title={Equivalent norms in spaces of entire functions of exponential type},
  author={Logvinenko, Vladimir and Sereda, Yu.},
  journal={Teor. Funkci{\i}Funkcional. Anal. i Prilozen. Vyp},
  volume={20},
  pages={102--111},
  year={1974}
}

@inproceedings{paneah61,
  title={Some theorems of {P}aley--{W}iener type},
  author={Paneah, Boris Petrovich},
  booktitle={Doklady Akademii Nauk},
  volume={138},
  number={1},
  pages={47--50},
  year={1961},
  organization={Russian Academy of Sciences}
}

@article{pruss84,
  title={On the spectrum of ${C}_0$-semigroups},
  author={Pr{\"u}ss, Jan},
  journal={Transactions of the American Mathematical Society},
  volume={284},
  number={2},
  pages={847--857},
  year={1984}
}

@article{rauch74,
  title={Exponential decay of solutions to hyperbolic equations in bounded domains},
  author={Rauch, Jeffrey and Taylor, Michael and Phillips, Ralph},
  journal={Indiana University Mathematics Journal},
  volume={24},
  number={1},
  pages={79--86},
  year={1974},
  publisher={JSTOR}
}

@book{zworski-book,
  title={Semiclassical analysis},
  author={Zworski, Maciej},
  volume={138},
  year={2012},
  publisher={American Mathematical Soc.}
}

@article{bourgain18,
  title={Spectral gaps without the pressure condition},
  author={Bourgain, Jean and Dyatlov, Semyon},
  journal={Annals of Mathematics},
  volume={187},
  number={3},
  pages={825--867},
  year={2018},
  publisher={JSTOR}
}

@Article{shubin98,
  title={Some harmonic analysis questions suggested by {A}nderson-{B}ernoulli models},
  author={Shubin, Carol and Vakilian, Ramin and Wolff, Thomas},
  journal={Geometric and Functional Analysis},
  volume={8},
  number={5},
  pages={932--964},
  year={1998},
  publisher={Springer}
}

\end{document}